\documentclass{amsart}

\title{Exact colinearity of centroids of iterated midpoint hexagons}
\author{Jack Edward Tisdell}
\address{McGill University, Montr\'eal, Quebec, Canada}
\email{jack.tisdell@mail.mcgill.ca}
\subjclass[2020]{51F, 51M04}
\keywords{Polygon iteration, centroid, finite Fourier series}

\usepackage{tikz}
\usetikzlibrary{calc}
\usepackage{pgfplots}
\pgfplotsset{
    compat=newest,
}
\usepackage{amsmath,amssymb,amsthm}
\usepackage{mathtools}

\newcommand*\C{\mathbb C}

\DeclarePairedDelimiter\abs\lvert\rvert

\newtheorem{theorem}{Theorem}
\newtheorem{proposition}[theorem]{Proposition}
\newtheorem{lemma}[theorem]{Lemma}

\theoremstyle{definition}

\usepackage[]{hyperref}

\begin{document}

\maketitle

\begin{abstract}
    We study the iteration that replaces a planar hexagon by the hexagon formed by joining the midpoints of consecutive edges. While this iteration quickly drives any polygon toward a point and their shapes asymptotically regularize, we show a stronger and unexpected rigidity holds for hexagons: from the second iterate onward, the centroids of the filled hexagons all lie exactly on a fixed line. This exact colinearity reflects a special algebraic feature of the hexagonal case and does not hold generally for any other polygons.
\end{abstract}

\section{Introduction}
\label{sec:intro}

We consider the following iteration on planar hexagons: given a hexagon, form a new one by joining the midpoints of consecutive edges, and repeat. This construction quickly drives any polygon toward the centroid of the original vertices and, after rescaling, the limiting shape regularizes in a sense. In this note, we show that a stronger and more rigid phenomenon occurs.\footnote{The author thanks Dao Thanh Oai for pointing out the colinearity phenomenon in a MathOverflow discussion: \url{https://mathoverflow.net/q/507514/490554}. This note is adapted from the author's answer there.}

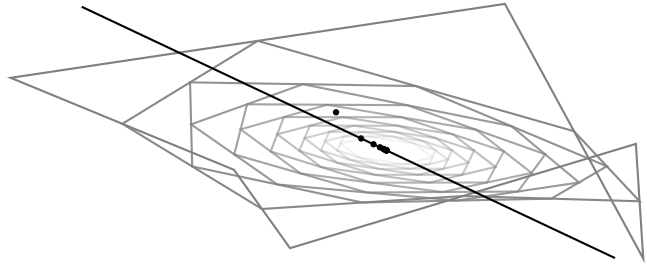
\begin{figure}[ht]
    \centering
    \begin{tikzpicture}[scale=1.5, rotate=15]

      \pgfmathsetmacro{\xA}{0.0}   \pgfmathsetmacro{\yA}{0.4}
      \pgfmathsetmacro{\xB}{3.2}   \pgfmathsetmacro{\yB}{0.5}
      \pgfmathsetmacro{\xC}{3.0}   \pgfmathsetmacro{\yC}{-.5}
      \pgfmathsetmacro{\xD}{2.4}   \pgfmathsetmacro{\yD}{2.0}
      \pgfmathsetmacro{\xE}{-2}   \pgfmathsetmacro{\yE}{2.5}
      \pgfmathsetmacro{\xF}{-0.3}  \pgfmathsetmacro{\yF}{1.2}

      \coordinate (C) at ({1/6*(\xA+\xB+\xC+\xD+\xE+\xF)},{1/6*(\yA+\yB+\yC+\yD+\yE+\yF)});

      \pgfmathsetmacro{\n}{13}
      \foreach \step in {0,...,\n}{
          \pgfmathsetmacro{\t}{100*\step/\n}

          \draw[gray, thick, opacity={1-.9/\n*\step}]
          (\xA,\yA)--(\xB,\yB)--(\xC,\yC)--(\xD,\yD)--(\xE,\yE)--(\xF,\yF)--cycle;

        \pgfmathsetmacro{\twiceA}{
          \xA*\yB-\xB*\yA +
          \xB*\yC-\xC*\yB +
          \xC*\yD-\xD*\yC +
          \xD*\yE-\xE*\yD +
          \xE*\yF-\xF*\yE +
          \xF*\yA-\xA*\yF
        }

        \pgfmathsetmacro{\Cxnum}{
          (\xA+\xB)*(\xA*\yB-\xB*\yA) +
          (\xB+\xC)*(\xB*\yC-\xC*\yB) +
          (\xC+\xD)*(\xC*\yD-\xD*\yC) +
          (\xD+\xE)*(\xD*\yE-\xE*\yD) +
          (\xE+\xF)*(\xE*\yF-\xF*\yE) +
          (\xF+\xA)*(\xF*\yA-\xA*\yF)
        }

        \pgfmathsetmacro{\Cynum}{
          (\yA+\yB)*(\xA*\yB-\xB*\yA) +
          (\yB+\yC)*(\xB*\yC-\xC*\yB) +
          (\yC+\yD)*(\xC*\yD-\xD*\yC) +
          (\yD+\yE)*(\xD*\yE-\xE*\yD) +
          (\yE+\yF)*(\xE*\yF-\xF*\yE) +
          (\yF+\yA)*(\xF*\yA-\xA*\yF)
        }

        \pgfmathsetmacro{\Cx}{\Cxnum/(3*\twiceA)}
        \pgfmathsetmacro{\Cy}{\Cynum/(3*\twiceA)}

        \coordinate (G\step) at (\Cx,\Cy);
        \fill (G\step) circle[radius=.8pt];

        \pgfmathsetmacro{\nxA}{(\xA+\xB)/2} \pgfmathsetmacro{\nyA}{(\yA+\yB)/2}
        \pgfmathsetmacro{\nxB}{(\xB+\xC)/2} \pgfmathsetmacro{\nyB}{(\yB+\yC)/2}
        \pgfmathsetmacro{\nxC}{(\xC+\xD)/2} \pgfmathsetmacro{\nyC}{(\yC+\yD)/2}
        \pgfmathsetmacro{\nxD}{(\xD+\xE)/2} \pgfmathsetmacro{\nyD}{(\yD+\yE)/2}
        \pgfmathsetmacro{\nxE}{(\xE+\xF)/2} \pgfmathsetmacro{\nyE}{(\yE+\yF)/2}
        \pgfmathsetmacro{\nxF}{(\xF+\xA)/2} \pgfmathsetmacro{\nyF}{(\yF+\yA)/2}

        \xdef\xA{\nxA} \xdef\yA{\nyA}
        \xdef\xB{\nxB} \xdef\yB{\nyB}
        \xdef\xC{\nxC} \xdef\yC{\nyC}
        \xdef\xD{\nxD} \xdef\yD{\nyD}
        \xdef\xE{\nxE} \xdef\yE{\nyE}
        \xdef\xF{\nxF} \xdef\yF{\nyF}
      }
      \draw[thick] ($(C)!12!(G1)$) -- ($(C)!-9!(G1)$);

    \end{tikzpicture}
    \caption{Starting with any planar hexagon and iteratively forming the hexagon by joining the midpoints of consecutive sides, all the centroids except possibly the first lie on a common line.}
    \label{fig:example}
\end{figure}

Let $P_0$ be any closed hexagon in the plane (no convexity, simplicity, or non-degeneracy assumptions are imposed) and let $P_{n+1}$ be obtained from $P_n$ by joining the midpoints of consecutive edges. Denote by $G_n$ the centroid of $P_n$ defined algebraically in terms of the vertex coordinates by the standard formula for polygon centroids. (In the case of simple polygons, this is the centroid of the filled polygon, but the algebraic definition extends to arbitrary polygons.) Our main result is the following.

\begin{theorem}
    From the second iterate onward (that is, excluding $G_0$), the points $G_n$ ($n\ge 1$) lie on a fixed line and converge eventually monotonically along it to the centroid of the vertices of $P_0$.
    \label{thm:centroid_colinearity}
\end{theorem}

An example is illustrated in Figure~\ref{fig:example}. The appearance of such exact colinearity is unexpected. While the midpoint map rapidly produces shapes that are \emph{nearly} affinely regular, this alone would suggest only approximate behavior.

By contrast, the centroids of iterated 1-, 2-, 3-, and 4-gons are constant after the first iteration for elementary reasons and are generally not colinear for $m$-gons when $m=5$ or $m\ge 7$. The exact colinearity phenomenon for $m=6$ is stronger than mere shape regularization, which occurs for all polygons.

\section{The midpoint iteration and its eigenstructure in the space of hexagons}
\label{sec:midpoint_map_&_space_of_hexagons}
Hereafter, ``the plane'' is the complex plane and points in the plane are complex numbers. An oriented hexagon in the plane is represented by its vertex vector $v = (v_0,\dots,v_5) \in \C^6$ and we think of the vector space $\C^6$ as the ``space of hexagons'' and refer to its elements as \emph{hexagons}. The space of hexagons includes all sorts of self-intersecting and degenerate hexagons. This is fine and indeed, some of these creatures will play a crucial role. For convenience, we shall index vector components $0,\dots,5$, as above, and indices are always to be understood modulo 6. For obvious reasons, we shall refer to the components $v_k$ of a hexagon $v \in \C^6$ as its \emph{vertices}. 

Let $M : \C^6 \to \C^6$ denote the \emph{midpoint map} which, given any hexagon $v \in \C^6$, produces the hexagon $Mv$ by joining the midpoints of consecutive edges of $v$. Explicitly, the vertices of $Mv$ are given by $(Mv)_k = \frac12(v_k+v_{k+1})$. Thus, $M$ is a linear map on the space of hexagons. The eigenvectors of $M$ are the discrete Fourier modes $e^{(j)} = (1,\omega^j,\omega^{2j},\dots,\omega^{5j})$ for $j=0,\dots,5$ where $\omega = e^{2\pi i/6}$ is a primitive sixth root of unity and the corresponding eigenvalues are $\lambda_j = \frac{1+\omega^j}{2}$. (This follows from the general theory of circulant matrices but one may verify it readily.)

It is useful to interpret the eigenvectors $e^{(j)}$, understood as hexagons, geometrically. They are (up to similarity) all and only the oriented regular hexagons, including degenerate ones. $e^{(0)} = (1,1,1,1,1,1)$ is the hexagon with six coincident vertices, $e^{(1)} = (1,\omega,\dots,\omega^5)$ and its complex conjugate $e^{(5)} = \bar e^{(1)}$ are the \textit{bona fide} regular hexagons in both orientations with vertices at the sixth roots of unity, $e^{(2)} = (1,\omega^2,\omega^4,1,\omega^2,\omega^4)$ and its conjugate $e^{(4)}=\bar e^{(2)}$ are the double-covered equilateral triangles in both orientations, and $e^{(3)} = (1,-1,1,-1,1,-1)$ is a multiple-covered segment. The eigenvalue relation $Me^{(j)} = \lambda_j e^{(j)}$ expresses the fact that each of these hexagons is similar to its midpoint hexagon.

Thus, any hexagon may be decomposed as a complex linear combination of the six oriented regular hexagons (including the degenerate ones). We will call the terms of such a decomposition the (Fourier) \emph{modes}.

We are interested in iterating the midpoint map. If $v = \sum_{j=0}^5 \xi_je^{(j)}$ where $\xi_j \in \C$ is any hexagon then $M^nv = \sum_j \xi_j\lambda_j^ne^{(j)}$. Some observations about this iteration are immediate from this decomposition. Since $\lambda_3 = 0$, the $e^{(3)}$ mode dies on the first step. This is very suggestive: looking ahead, if the colinearity phenomenon holds in the $\xi_3=0$ case, this would explain the exceptional initial centroid. Secondly, $\lambda_0 = 1$ and $\abs{\lambda_j} < 1$ for $j\ne 0$ so the iterates converge exponentially quickly to the invariant $e^{(0)}$ mode $\xi_0e^{(0)}$. Varying the coefficient $\xi_0$ amounts to translating the hexagon $v$ in the plane. Moreover, since the sum of the (second, third, sixth) roots of unity is zero, the sum of the vertices of $v$ is $\sum_k v_k = 6\xi_0$ so $\xi_0 = \frac{v_0+\dots+v_k}{6}\in \C$ is the \emph{vertex centroid} (not to be confused with the centroid of the filled polygon) of $v$ (and of $M^nv$ for all $n$). Finally, it can be shown that every hexagon of the form $\xi_1e^{(1)} + \xi_5e^{(5)}$ is a linear image of the (\textit{bona fide}) regular hexagon. Since $\abs{\lambda_2}=\abs{\lambda_4}=\frac12 < \abs{\lambda_1}=\abs{\lambda_5} = \frac{\sqrt 3}{2}$, after translating so that $\xi_0=0$ and then scaling by $(\frac{2}{\sqrt 3})^n$, we see that that $(\frac{2}{\sqrt 3})^nM^nv$ is nearly affinely regular (if $\xi_1$ and $\xi_5$ are not both zero). One can show further that the limiting shapes alternate between two fixed affine images of the regular hexagon, up to an exponential error, but this is incidental to the exact colinearity phenomenon we wish to explain.

The Fourier decomposition of polygons in the complex plane used here is standard in the study of iterated maps on polygons. See, for example, \cite{Douglas1940,Schoenberg1950,Neumann1941}.

\section{The centroids}

The centroid $G(v) \in \C$ of a hexagon $v \in \C^6$ is given in terms of the vertices $v_k$ by
\begin{equation}
    G(v) = \frac{Z(v)}{6A(v)},
    \;\;
    A(v) = \frac12\sum_k \Im(\bar v_kv_{k+1}),
    \;\;
    Z(v) = \sum_k (v_k+v_{k+1})\Im(\bar v_kv_{k+1})
    \label{eq:shoelace}
\end{equation}
where $\Im(z)$ is the imaginary part of $z\in\C$ and $\bar z$ is the complex conjugate of $z$

The real quantity $A(v)$ is the signed area of the hexagon $v$. (If one expands $v_k = x_k+iy_k$, one sees that the above formula for $A(v)$ coincides with the ``shoelace formula'' $A(v) = \frac12\sum_k (x_ky_{k+1}-x_{k+1}y_k)$ and similarly, the real and imaginary parts of $Z(v)$ in terms of $x_k,y_k$ agree with the usual formulas.) This coincides with the centroid of the filled polygon in the case of simple hexagons and but the formulas apply to arbitrary hexagons. For intuition, the reader may keep in mind the case of simple convex hexagons. As noted above, the iteration quickly falls into the simple convex case anyway (except in the degenerate case $\xi_1=\xi_5=0$ which is geometrically even simpler, in fact).

It is worth underscoring that $G(v)$ is highly non-linear in $v$. This is why the rigid exact colinearity phenomenon is surprising. Having defined $G(v)$, we can state the key lemma behind the result.

\begin{lemma}
    If $v = \sum_j \xi_je^{(j)}\in \C^6$ is any hexagon with $\xi_0=\xi_3=0$, then $Z(Mv) = \frac38 Z(v)$ where $Z$ is as defined in \eqref{eq:shoelace}.
    \label{lemma:real_centroid_scaling}
\end{lemma}

\begin{proof}
    Let $v = \sum_j \xi_je^{(j)}$ with $\xi_0=\xi_3=0$. First, 
    \begin{equation}
        Z(v) = \sum_k (v_k+v_{k+1})\Im(\bar v_kv_{k+1}) = \frac1{2i}\sum_k(v_k+v_{k+1})(\bar v_kv_{k+1}-v_k\bar v_{k+1}).
        \label{eq:Z_expansion}
    \end{equation}
    After expanding the product on the right, consider the term $v_{k+1}\bar v_kv_k$. Substituting $v_k = \sum_j \xi_je^{(j)}_k = \sum_j \xi_j\omega^{jk}$, we have
    $
        v_{k+1}\bar v_kv_k
        = \sum_{p,q,r} \xi_p\bar\xi_q\xi_r\omega^{p(k+1)}\omega^{-qk}\omega^{rk}
        = \sum_{p,q,r} \xi_p\bar\xi_q\xi_r\omega^p\omega^{(p-q+r)k}.
    $
    Similarly for $v_k\bar v_{k+1}v_{k+1}$, $v_{k+1}\bar v_kv_{k+1}$, $v_k\bar v_{k+1}v_k$ but with the factor of $\omega^p$ replaced by $\omega^{r-q}$, $\omega^{p+r}$, $\omega^{-q}$, respectively. Altogether, $(v_k+v_{k+1})(\bar v_kv_{k+1}-v_k\bar v_{k+1}) = \sum_{p,q,r} \xi_p\bar\xi_q\xi_r(\omega^p-\omega^{r-q}+\omega^{p+r}-\omega^{-q})\omega^{(p-q+r)k}$.
    Substituting into \eqref{eq:Z_expansion} and summing over $k$, the only terms which survive are those for which $p-q+r \equiv 0 \pmod 6$. But in that case, $r-q\equiv -p$ and $p+r\equiv q$ so
    \[
        Z(v) 
        = \frac{6}{2i}\sum_{p-q+r\equiv 0} \xi_p\bar\xi_q\xi_r(\omega^p-\omega^{-p}+\omega^q-\omega^{-q})
        = 6\sum_{p,q} \xi_p\bar\xi_q\xi_{q-p}\Im(\omega^p+\omega^q).
    \]
    By assumption, $\xi_0=\xi_3 = 0$ so all terms in which $p$, $q$, or $q-p$ are $0$ or $3$ vanish. Moreover, $\Im(\omega^p+\omega^q)=0$ whenever $p+q\equiv0\pmod 6$. This leaves only the four terms $(p,q) \in \{(1,2),(5,4),(4,5),(2,1)\}$ and in each case, $\Im(\omega^p+\omega^q)=\sqrt{3}$. Thus,
    \begin{equation}
        Z(v)
        = 6\sqrt 3\big( \xi_1\bar\xi_2\xi_1-\xi_5\bar\xi_4\xi_5 + \xi_4\bar\xi_5\xi_1 - \xi_2\bar\xi_1\xi_5 \big).
        \label{eq:Z_fourier}
    \end{equation}
    Since \eqref{eq:Z_fourier} holds for any $v = \sum_j\xi_je^{(j)}$ with $\xi_0=\xi_3=0$, it applies also to $Mv = \sum_j\lambda_j\xi_je^{(j)}$ replacing $\xi_j$ by $\lambda_j\xi_j$. By direct calculation, $\lambda_1\bar\lambda_2\lambda_1 = \lambda_5\bar\lambda_4\lambda_5 = \lambda_4\bar\lambda_5\lambda_1 = \lambda_2\bar\lambda_1\lambda_5 = \frac38$. Thus, $Z(Mv) = \frac38Z(v)$. 
\end{proof}

Most of the above proof amounts to computing $Z(v)$ and $Z(Mv)$ in terms of the Fourier coefficients $\xi_j$ and $\lambda_j\xi_j$ and observing the abundant cancellation. The real miracle though occurs in the very last line: all the surviving terms include a triple product $\lambda_p\bar\lambda_q\lambda_{q-p}$ of eigenvalues all of which have the \emph{same} real value $\frac38$. (Such products are all real, it's that they are all equal which is remarkable.)

This lemma already explains the colinearity. Start with hexagon $v=\sum_j\xi_je^{(j)}$.  Without loss, we may assume $\xi_0=0$. The first iteration $Mv$ kills $\xi_3$. Thereafter, the numerators $Z(M^nv)$ in $G(M^nv)$ pick up a real factor of $\frac38$ at each iteration and the denominators $6A(M^nv)$ are all real. So $G(M^nv)$ lie on a common line through the origin for all $n\ge 1$. (In general, a common line through $\xi_0$.) But we can say more:

\begin{theorem}
    Let $v = \sum_j \xi_j e^{(j)}$ be any hexagon with $\xi_0=\xi_3=0$. Then
    \begin{equation}
        G(M^nv) = \frac{1}{2^n}\Big(9\sqrt 3\big(\abs{\xi_1}^2-\abs{\xi_5}^2+3^{-n}(\abs{\xi_2}^2-\abs{\xi_4}^2)\big)\Big)^{-1}Z(v)
        \label{eq:centroid_fourier}
    \end{equation}
    In particular, $G(M^nv)$ is a real multiple of $G(v) = Z(v)/6A(v)$ for every $n$, thus they all lie on a common line through the origin. Moreover, the convergence along this line is eventually monotonic. 
    \label{thm:centroid_fourier}
\end{theorem}
\begin{proof}
    We may express $A(v) = \frac12\sum_k\Im(\bar v_kv_{k+1})$ in terms of the coefficients $\xi_j$ as we did $Z(v)$ in the proof of Lemma~\ref{lemma:real_centroid_scaling}. Here, $\bar v_kv_{k+1} = \sum_{p,q}\bar\xi_p\xi_q\omega^q\omega^{(q-p)k}$. Summing over $k$, only the diagonal terms terms survive so $A(v) = \frac12\Im(\sum_k\bar v_kv_{k+1}) = \frac12\Im(6\sum_j \bar\xi_j\xi_j\omega^j) = 3\sum_j\abs{\xi_j}^2\Im(\omega^j) = \frac{3\sqrt 3}{2}(\abs{\xi_1}^2-\abs{\xi_5}^2+\abs{\xi_2}^2-\abs{\xi_3}^2)$. Then, $A(M^nv)$ is obtained by replacing each $\xi_j$ with $\lambda_j\xi_j$. This together with Lemma~\ref{lemma:real_centroid_scaling} and $\abs{\lambda_1}^2 = \abs{\lambda_5}^2 = \frac34$ and $\abs{\lambda_2}^2 = \abs{\lambda_4}^2 = \frac14$ yields \eqref{eq:centroid_fourier}. Eventual monotonicity follows from the fact that the coefficient in \eqref{eq:centroid_fourier} changes sign at most once as $n$ increases.
\end{proof}

The assumption $\xi_0=\xi_3=0$ is not restrictive. Without loss, we may always assume $\xi_0 = 0$ by first translating so that the vertex centroid---which is fixed under the midpoint map---is at the origin. Then the first iteration kills $\xi_3$ whence Theorem~\ref{thm:centroid_fourier} applies, thereby proving Theorem~\ref{thm:centroid_colinearity}.

\section{The colinearity phenomenon is unique to hexagons}
Does the phenomenon occur for $m$-gons? We claim that in all cases except $m=6$, the colinearity generally fails or else the centroid sequence is constant for elementary reasons.

\begin{proposition}
    The centroids of $m$-gons under the midpoint iteration are invariant after the first step for $m=1,2,3,4$. For $m=5$ and $m\ge 7$, there are instances in which no two of the iterated centroids are colinear with their limit point. (In particular, they are not eventually colinear.)
    \label{prop:m-gon_behavior}
\end{proposition}

\begin{proof}
    The cases $m=1$ and $m=2$ are trivial. For $m=3$, the centroid of a filled triangle coincides with its vertex centroid which is invariant under the midpoint map by direct calculation. For $m=4$, by Varignon's Theorem, the midpoint polygon of any quadrilateral is a (possibly degenerate) parallelogram (by an elementary geometric proof: each of a pair of opposite sides of the midpoint quadrilateral is parallel to the diagonal and half its length), whence the centroid is invariant after the first iteration.

    To see the failure of colinearity for all other $m\ne6$, we again engage with the algebra. Now, $\C^m$ is the space of $m$-gons. As before, $(Mv)_k = \frac12(v_k+v_{k+1})$. The eigenvectors of $M$ are $e^{(j)} = (1,\omega^j,\omega^{2j},\dots,\omega^{(m-1)j})$ for $j=0,\dots,m-1$ where now $\omega = e^{2\pi i/m}$ and the corresponding eigenvalues are $\lambda_j = \frac{1+\omega^j}{2}$. The centroid $G(v) = Z(v)/6A(v)$ is defined exactly as before. Forgive the abuse of notation: in this section, $\omega, e^{(j)}, \lambda_j$ are understood relative to whichever $m$ is under consideration.

    As in the proof of Lemma~\ref{lemma:real_centroid_scaling}, we find $Z(v) = m\sum_{p,q}\xi_p\bar\xi_q\xi_{q-p}\Im(\omega^p+\omega^q)$ for $v = \sum_j\xi_je^{(j)}$ with indices understood mod $m$. Of course, $Z(M^nv)$ is then obtained by replacing $\xi_j$ with $\lambda_j^n\xi_j$ and thus, the $(p,q)$ term in the expansion picks up a factor of $(\lambda_p\bar\lambda_q\lambda_{q-p})^n$.

    Let $m\ge 7$ and consider an $m$-gon $v \in \C^m$ of the form $v = \xi_1e^{(1)}+\xi_2e^{(2)}+\xi_3e^{(3)}$. For such $v$, the only possibly non-zero terms in the $Z(v)$ expansion come from $(p,q) \in \{(1,2),(1,3),(2,3)\}$. Write $\mu=\lambda_1\bar\lambda_2\lambda_1$ and $\nu=\lambda_1\bar\lambda_3\lambda_2 = \lambda_2\bar\lambda_3\lambda_1$. By direct calculation $\lambda_p\bar\lambda_q\lambda_{q-p} = \frac14\Re(1+\omega^p+\omega^q+\omega^{q-p})$ where $\Re(z)$ is the real part of $z$ and, as one can check using $\Re(\omega^j)=\cos\frac{2\pi j}{m}$ and the double- and triple-angle formulas, $\nu/\mu = 2\cos\frac{2\pi}m-1$. Thus, for the particular $m$-gon $v_* = ie^{(1)} - e^{(2)}+e^{(3)}$, we have $M^nv_* = i\lambda_1^ne^{(1)} - \lambda_2^ne^{(2)} + \lambda_3^ne^{(3)}$ and $Z(M^nv_*) = m(\mu^n\Im(\omega+\omega^2) - \nu^ni\Im(\omega+\omega^3) - \nu^ni\Im(\omega^2+\omega^3))$. The centroids $G(M^nv_*)$ converge to the origin, so if they were (eventually) colinear, their line would have to pass through the origin. But $Z(M^nv_*)$ (and hence $G(M^nv_*)$) is on the line through the origin of slope $\frac{\Im Z(M^nv_*)}{\Re Z(M^nv_*)} = \theta(\nu/\mu)^n$ for the (real) constant $\theta = -\Im(\omega+\omega^2+2\omega^3)/\Im(\omega+\omega^2) \ne 0$. Since $\nu/\mu = 2\cos\frac{2\pi}{m}-1 \in (0,1)$ (as $m\ge 7$), these lines are all distinct.

    The pentagonal case $m=5$ is similar but the conjugacy of $\omega^2$ and $\omega^3$ upsets the above analysis. In this case, for $v\in\C^5$ of the form $v = \xi_1e^{(1)}+\xi_2e^{(2)}+\xi_3e^{(3)}$, the non-vanishing terms in the expansion of $Z(v)$ are $(p,q)\in \{(1,2),(1,3),(3,1)\}$. Let $\mu = \lambda_1\bar\lambda_2\lambda_1$ and $\nu =\lambda_1\bar\lambda_3\lambda_2=\frac14\Re(1+\omega+\omega^2+\omega^3)=\frac14\Re(1+\omega+\omega^3+\omega^3)=\lambda_3\bar\lambda_1\lambda_3$. Here, for $v_*=ie^{(1)}-e^{(2)}+e^{(3)}$, we have $Z(M^nv_*) = 5(\mu^n\Im(\omega+\omega^2)-2i\nu^n\Im(\omega+\omega^3))$ so $Z(M^nv_*)$ (and hence $G(M^nv_*)$) lies on the line through the origin of slope $\frac{\Im Z(M^nv_*)}{\Re Z(M^nv_*)} = \tilde\theta(\nu/\mu)^n$ where $\tilde\theta = -2\Im(\omega+\omega^3)/\Im(\omega+\omega^2) \ne 0$. Since $\nu/\mu = 2\cos\frac{2\pi}{5}-1 \in (-1,0)$, these lines are all distinct.
\end{proof}

\noindent{\scshape Acknowledgments.} The author thanks Dao Thanh Oai for pointing out the colinearity phenomenon in a MathOverflow discussion: \url{https://mathoverflow.net/q/507514/490554}. This note is adapted from the author's answer there.

\bibliographystyle{vancouver}
\bibliography{references}

\end{document}